\documentclass{article}
\usepackage{amsmath}
\usepackage{amsfonts}
\usepackage{amsthm}
\usepackage{amssymb}
\usepackage{color}

\newtheorem{theorem}{Theorem}[section]

\newtheorem{proposition}[theorem]{Proposition}

\newtheorem{problem}[theorem]{Problem}
\newtheorem{example}[theorem]{Example}
\newtheorem{corollary}[theorem]{Corollary}

\newtheorem{claim}[theorem]{Claim}

\theoremstyle{definition}
\newtheorem{definition}[theorem]{Definition}
\theoremstyle{remark}
\newtheorem{remark}[theorem]{Remark}

\newcommand{\rnp}{{\rm RNP}}
\newcommand{\RNP}{{\rm RNP}}

\newcommand{\conv}{{\rm conv}\hskip0.02cm}

\def\f2{\mathbb{F}_2}

\def\dist{\hskip0.02cm{\rm dist}\hskip0.01cm}

\newcommand{\ep}{\varepsilon}

\begin{document}

\title{\LARGE On metric characterizations of the Radon-Nikod\'ym and related properties of Banach
spaces}

\author{Mikhail Ostrovskii\footnote{Supported in part by NSF
DMS-1201269}}

\date{\today}
\maketitle

\noindent{\bf Abstract} We find a class of metric structures which
do not admit bilipschitz embeddings into Banach spaces with the
Radon-Nikod\'ym property. Our proof relies on Chatterji's (1968)
martingale characterization of the RNP and does not use the
Cheeger's (1999) metric differentiation theory. The class includes
the infinite diamond and both Laakso (2000) spaces. We also show
that for each of these structures there is a non-RNP Banach space
which does not admit its bilipschitz embedding.

We prove that a dual Banach space does not have the \RNP\ if and
only if it admits a bilipschitz embedding of the infinite diamond.

The paper also contains related characterizations of reflexivity
and the infinite tree property.

\medskip

\noindent{\bf Keywords:} Banach space, diamond graph, geodesic,
infinite tree property, Laakso space, martingale, Radon-Nikod\'ym
property, reflexivity
\medskip

\noindent{\bf 2010 Mathematics Subject Classification.} Primary:
46B22; Secondary: 05C12, 30L05, 46B10, 46B85, 54E35.
\medskip

\begin{large}

\tableofcontents

\section{Introduction and some general problems}\label{S:MetrChar}

In the recent work on metric embeddings a substantial role is
played by existence and non-existence of bilipschitz embeddings of
metric spaces into Banach spaces with the Radon-Nikod\'{y}m
property (RNP, for short), see \cite{CK06,CK09,LN06}. At the
seminar ``Nonlinear geometry of Banach spaces'' (Texas A \&\ M
University, August 2009) Johnson suggested the problem of metric
characterization of reflexivity and the Radon-Nikod\'{y}m property
 \cite[Problem 1.1]{Tex09}.
\medskip

The problem of metric characterization of the Radon-Nikod\'{y}m
property and reflexivity can be understood and approached in
several different ways. The purpose of the present paper is to
develop one of the possible approaches to it. Our approach is
similar to the approach of metric characterization of
superreflexivity suggested by Bourgain \cite{Bou86} in the first
paper on metric characterizations of classes of Banach spaces.
This approach (for different classes of spaces) was later followed
in \cite{Bau07,BMW86,JS09,MN08,Ost11,Ost13a,Ost13+,Pis86} (see
also accounts in \cite{Pis11} and \cite{Ost13b}). The mentioned
approach is based on the notion of test spaces.

\begin{definition}\label{D:TestSp} Let $\mathcal{P}$ be a class of Banach spaces
and let $T=\{T_\alpha\}_{\alpha\in A}$ be a set of metric spaces.
We say that $T$ is a set of {\it test-spaces} for $\mathcal{P}$ if
the following two conditions are equivalent:
\begin{enumerate}
\item $X\notin\mathcal{P}$.

\item\label{I:DefTesSp2} The spaces $\{T_\alpha\}_{\alpha\in A}$
admit bilipschitz embeddings into $X$ with uniformly boun\-ded
distortion.
\end{enumerate}
\end{definition}

\begin{remark} We write $X\notin\mathcal{P}$ rather than
$X\in\mathcal{P}$ for terminological reasons: we would like to use
terms ``test-spaces for reflexivity, superreflexivity, etc.''
rather than ``test-spaces for {\bf non}reflexivity, {\bf
non}superreflexivity, etc.''
\end{remark}

We will be mostly interested in the following special case of
Definition \ref{D:TestSp}:

\begin{definition} We say that a metric space $X$ is a \emph{test
space} for $\mathcal{P}$ if the bilipschitz embeddability of $X$
into a Banach space $Y$ is equivalent to $Y\notin\mathcal{P}$.
\end{definition}

The following problems are open.

\begin{problem} Does there exist a test space for the \RNP?
\end{problem}

\begin{problem}\label{P:MetTSRefl} Does there exist a test space for
reflexivity?
\end{problem}

\begin{remark}
It should be mentioned that, as we know from the well-known
example of Ribe \cite{Rib84} (see also \cite[Theorem
10.28]{BL00}), the \RNP\ and reflexivity are not preserved by
uniform homeomorphisms and therefore their metric
characterizations are not included into the Ribe program. See
\cite{Bal12} and \cite{Nao12} for description of the Ribe program.
\end{remark}

\begin{remark} The example of Ribe \cite{Rib84} mentioned above, combined with the
classical observation of \cite{CK63} (see also \cite[Proposition
1.11]{BL00} and \cite[Lemma 9.7]{Ost13b}) that uniformly
continuous maps between Banach spaces are Lipschitz for ``large''
distances, implies that RNP and reflexivity cannot be
characterized by uniformly discrete test spaces.
\end{remark}

The natural candidates for being test spaces for reflexivity and
the \RNP\ are the infinite diamond $D_\omega$ and the first Laakso
space $L_\omega$ (we recall the definition of $D_\omega$ below).
These spaces are the natural candidates because it was shown in
\cite{JS09} (see also \cite{Ost11} and \cite[Section
9.3.2]{Ost13b}) that their finite versions form collections of
test spaces for superreflexivity. However. it turns out that these
natural candidates for being test spaces for RNP are not such.
More precisely if a Banach space $X$ admits a bilipschitz
embedding of the infinite diamond or the first Laakso space, then
$X\notin\RNP$, but there are non-\RNP\ Banach spaces which do not
admit bilipschitz embedding of these spaces. In the case of the
infinite diamond these statements were proved in \cite{Ost11}. In
the case of the first Laakso space, only the first statement was
proved in \cite{Ost11}. One of the purposes of this paper is to
generalize these results. In Theorems \ref{T:MetrNonRNP} and
\ref{T:IsoMNonRNP} we find a wide class $\mathcal{R}$ of metric
spaces containing, in addition to $D_\omega$ and $L_\omega$, also
the second Laakso space $X_\omega$ (defined in Example
\ref{Ex:LaaksoAhlforsRegularPI}), and such that the bilipschitz
embeddability of a metric space $M\in\mathcal{R}$ into a Banach
space $X$ implies that $X\notin\RNP$. On the other hand, we prove
(Theorem \ref{T:NotTestSp}) that for each $M\in\mathcal{R}$ there
exists a non-\RNP\ space $X$ which does not admit bilipschitz
embeddings of $M$.\medskip

Here we would like to mention that Cheeger and Kleiner
\cite[Corollary 1.9]{CK09} used the theory of differentiability of
functions on metric spaces developed by Chee\-ger \cite{Che99}
(see also \cite{Kei04,KM11}) in order to show that the Laakso
spaces $L_\omega$ and $X_\omega$ (introduced in \cite{Laa00}, see
\cite[p.~290]{LP01} and \cite[Example 1.2 and Example 1.4]{CK13}
for their elegant description) do not admit bilipschitz embeddings
into a Banach space with the \RNP. Theorem \ref{T:MetrNonRNP}
provides a different proof of non-embeddability of the Laakso
spaces $L_\omega$ and $X_\omega$ into Banach spaces with the \RNP.
See Section \ref{S:SecondLaakso} for the proof for $X_\omega$, the
$L_\omega$ case was already considered in \cite{Ost11}. This proof
does not use the theory of differentiability of functions on
metric spaces developed by Cheeger \cite{Che99}.\medskip

Now we recall the definition of infinite diamond. The {\it diamond
graph} of level $0$ is denoted $D_0$. It has two vertices joined
by an edge of length $1$. $D_n$ is obtained from $D_{n-1}$ as
follows. Each edge of $D_{n-1}$ is of length $2^{-(n-1)}$. Given
an edge $uv\in E(D_{n-1})$, it is replaced by a quadrilateral $u,
a, v, b$ with edge lengths $2^{-n}$. We endow $D_n$ with their
shortest path metrics. We consider the vertex of $D_n$ as a subset
of the vertex set of $D_{n+1}$, it is easy to check that this
defines an isometric embedding. We introduce $D_\omega$ as the
union of the vertex sets of $\{D_n\}_{n=0}^\infty$. For $u,v\in
D_\omega$ we introduce $d_{D_\omega}(u,v)$ as $d_{D_n}(u,v)$ where
$n\in\mathbb{N}$ is any integer for which $u,v\in V(D_n)$. Since
the natural embeddings $D_n\to D_{n+1}$ are isometric,
$d_{D_n}(u,v)$ does not depend on the choice of $n$ for which
$u,v\in V(D_n)$. To the best of my knowledge the first paper in
which diamond graphs $\{D_n\}_{n=0}^\infty$ were used in Metric
Geometry is \cite{GNRS04} (conference version was published in
1999).

\begin{definition}\label{D:tree}  Let $\delta>0$.
A sequence $\{x_i\}_{i=1}^\infty$ in a Banach space $X$ is called
a \emph{$\delta$-tree} \index{tree @ $\delta$-tree} if
$x_i=\frac12(x_{2i}+x_{2i+1})$ and
$||x_{2i}-x_i||=||x_{2i+1}-x_i||\ge\delta$. We say that a Banach
space $X$ has the \emph{infinite tree property} if it contains a
bounded $\delta$-tree for some $\delta>0$.
\end{definition}

\begin{theorem}[\cite{Ost11}]\label{T:BulgDiamond} A bilipschitz embeddability of $D_\omega$
into a Banach space $Y$ implies the infinite tree property of $Y$.
\end{theorem}

The mentioned above results about $D_\omega$ can be obtained by
combining Theorem \ref{T:BulgDiamond} with known results on the
\RNP. Namely, it is known \cite[page~111]{BL00} that Banach spaces
with the infinite tree property do not have the \rnp. On the other
hand, Bourgain and Rosenthal \cite{BR80} (see also \cite[Example
5.30]{BL00}) constructed an example of a Banach space without the
\RNP\ which does not have the infinite tree property.\medskip

In view of Theorem \ref{T:BulgDiamond} the following result which
we prove in this paper could be considered as a strengthening of a
result of Stegall \cite{Ste75}, who proved that dual Banach spaces
without the RNP have the infinite tree property.

\begin{theorem}\label{T:ExtStegall} A dual Banach space does not have the RNP if and
only if it admits a bilipschitz embedding of $D_\omega$.
\end{theorem}

This strengthening is not immediate because at the moment it is
not known whether the infinite tree property of a Banach space $Y$
implies the bilipschitz embeddability of $D_\omega$ into $Y$. In
this connection we observe that if we widen the notion of a test
space to what we call a \emph{submetric test-space}, we easily get
a characterization of the infinite tree property. We mean
following definition.

\begin{definition} A \emph{submetric test-space} for a class $\mathcal{P}$ of Banach spaces is defined as a metric space $T$
with a marked subset $S\subset T\times T$ such that the following
conditions are equivalent for a Banach space $X$:

\begin{enumerate}

\item $X\notin\mathcal{P}$.

\item There exist a constant $0<C<\infty$ and an embedding $f:T
\to X$ satisfying the condition

\begin{equation}\label{E:PartBilipsch}
\forall (x,y)\in S\quad d_{T}(x,y)\le ||f(x)-f(y)||\le C
d_{T}(x,y).
\end{equation}
\end{enumerate}

An embedding satisfying \eqref{E:PartBilipsch} is called a
\emph{partially bilipschitz} embedding. Pairs $(x,y)$ belonging to
$S$ are called \emph{active}.
\end{definition}

\begin{theorem}\label{T:SubInfTree} The class of Banach spaces with the infinite tree
property admits a submetric characterization in terms of the
metric space $D_\omega$ with the set of active pairs defined as
follows: a pair is active if and only if it is a pair of vertices
of a quadrilateral introduced in one of the steps.
\end{theorem}

In this paper we find a submetric test space for reflexivity
(Theorem \ref{T:SubMetRef}). As we have mentioned above (Problem
\ref{P:MetTSRefl}) the problem of existence of a \emph{metric}
test space for reflexivity remains open.

\section{Dual non-RNP spaces, proof of Theorem
\ref{T:ExtStegall}}\label{S:DualNonRNP}

\begin{proof} It is well-known (see \cite[Theorems A and
2]{Ste75}) that a dual Banach space $X^*$ does not have the \RNP\
if and only if $X$ contains a separable subspace $Y$ such that
$Y^*$ is nonseparable. First we prove this result in the case
where $X$ is separable. We use the construction of Stegall
\cite[Theorem 1]{Ste75} (see also \cite[pp.~192--195]{DU77}). In
the case where $X$ is a separable Banach space and $X^*$ is a
nonseparable Banach space, and $\ep>0$, he constructed

\begin{itemize}

\item A collection $\{x_{n,i}\}_{n=0,~i=0}^{\infty\hskip0.45cm
2^n-1}$ of vectors in $X$ satisfying $||x_{n,i}||<1+\ep$.

\item A collection $\{W_{n,i}\}_{n=0,~i=0}^{\infty\hskip0.45cm
2^n-1}$ of nonempty weak$^*$ compact convex subsets in $B_{X^*}$
(the unit ball of $X^*$) such that

\begin{equation}\label{E:dyadic}
W_{n,i}\supset W_{n+1,2i}\cup W_{n+1,2i+1}
\end{equation}

\begin{equation}\label{E:disjoint}
\{W_{n,i}\}_{i=0}^{2^n-1}~\hbox{ are pairwise disjoint.}
\end{equation}
\end{itemize}

These collections are such that if by $\Delta$ we denote the set
\[\bigcap_{n=0}^\infty\left(\bigcup_{i=0}^{2^n-1}
W_{n,i}\right),\] by $C(\Delta)$ we denote the space of weak$^*$
continuous functions on $\Delta$, by $h_{n,i}$ denote the
indicator function of $W_{n,i}\cap\Delta$, and by $T:X\to
C(\Delta)$ denote the natural embedding, the following condition
holds:
\begin{equation}\label{E:Stegall}
\sum_{n=0}^\infty\sum_{i=0}^{2^n-1}||Tx_{n,i}-h_{n,i}||<\ep.
\end{equation}
Observe that the function $h_{n,i}$ is continuous on $\Delta$
because of the conditions \eqref{E:dyadic} and \eqref{E:disjoint}.
\medskip

First we use these collections to construct a bounded
$\delta$-tree in $X^*$. The existence of such tree is well known,
see \cite[p.~114]{BL00}. We present the details because some of
the specific properties of our construction are needed for the
embedding of $D_\omega$.
\medskip

We pick a vector (arbitrarily) in each of the sets $\{\Delta\cap
W_{n,i}\}_{i=0}^{2^n-1}$. To conform with our notation for trees
we denote the vector picked in $\Delta\cap W_{n,i}$ by $y^n_j$,
where $j=i+2^n$, so $j=2^n,\dots,2^{n+1}-1$. We define $y^n_j$
with $j=2^{n-1},\dots,2^{n}-1$ by
\begin{equation}\label{E:defy_j}
y_j^n=\frac12(y^n_{2j}+y^n_{2j+1}).
\end{equation}
Next we define $y^n_j$ with $j=2^{n-2},\dots,2^{n-1}-1$ using
\eqref{E:defy_j}. We continue in an obvious way and define all
$\{y^n_j\}$ for $j=1,\dots,2^{n+1}-1$.\medskip

Our next purpose is to show that for each $j=1,\dots,2^{n}-1$ the
condition
\begin{equation}\label{E:x-sep}
|(y^n_{2j}-y^n_{2j+1})(x_{k+1,2r})|>1-2\ep
\end{equation}
holds, where $k$ and $r$ are determined by the condition:
$j=2^k+r$ with  $r<2^k$. To show this we observe, by using
\eqref{E:defy_j}, that $y_{2j}^n$ is a convex combination of
\begin{equation}\label{E:coll1}
y_m^n\quad\hbox{for }m=2^n+2^{n-k}r,
2^n+2^{n-k}r+1,\dots,2^n+2^{n-k}r+2^{n-k-1}-1.
\end{equation}
On the other hand, $y_{2j+1}^n$ for $j=2^k+r$ with  $r<2^k$ is a
convex combination of
\begin{equation}\label{E:coll2}
y_m^n\quad\hbox{for }m=2^n+2^{n-k}r+2^{n-k-1},
2^n+2^{n-k}r+2^{n-k-1}+1,\dots,2^n+2^{n-k}r+2^{n-k}-1.
\end{equation}

Observe that elements of \eqref{E:coll1} are contained in
$\Delta\cap W_{k+1,2r}$ and elements of \eqref{E:coll2} are
contained in $\Delta\cap W_{k+1,2r+1}$. Therefore $h_{k+1,2r}$
satisfies $h_{k+1,2r}(z)=1$ for each $z$ in \eqref{E:coll1} and
$h_{k+1,2r}(z)=0$ for each $z$ in \eqref{E:coll2}. Applying
\eqref{E:Stegall} we get
\begin{equation}\label{E:ineq1}
z(x_{k+1,2r})>1-\ep
\end{equation}
for each $z$ in \eqref{E:coll1} and
\begin{equation}\label{E:ineq2}
|z(x_{k+1,2r})|<\ep
\end{equation}
for each $z$ in \eqref{E:coll2}. It is clear that \eqref{E:ineq1}
and \eqref{E:ineq2} continue to hold if $z$ is a convex
combination of vectors in \eqref{E:coll1} and \eqref{E:coll2},
respectively. Inequality \eqref{E:x-sep} follows.

We construct such finite sequences $\{y^n_j\}_{j=1}^{2^{n+1}-1}$
for each $n=0,1,2,\dots$. We introduce $\{y_j\}_{i=1}^\infty$ by
$y_j=w^*-\lim_{n\to\infty} y_j^n$. It is clear that
\eqref{E:x-sep} implies
\begin{equation}\label{E:x-sep2}
|(y_{2j}-y_{2j+1})(x_{k+1,2r})|\ge 1-2\ep.
\end{equation}
Therefore $||y_{2j}-y_{2j+1}||\ge\frac{1-2\ep}{1+\ep}$ and the
sequence $\{y_j\}_{i=1}^\infty$ forms a $\delta$-tree with
$\delta=\frac{1-2\ep}{2(1+\ep)}$. It is also clear that this tree
is contained in the unit ball of $X^*$.
\medskip

Observe that using \eqref{E:Stegall} in the case where $n=i=0$ we
get $z(x_{0,0})\ge 1-\ep$ for each $z\in \Delta$. Thus
$y_j(x_{0,0})\ge 1-\ep$ and
\begin{equation}\label{E:ybelow}
 ||y_j||\ge\frac{1-\ep}{1+\ep}\quad\hbox{for
each }~j\in\mathbb{N}. \end{equation}

We need to derive one more consequence of inequalities
\eqref{E:ineq1} and \eqref{E:ineq2} and our construction. We
define the \emph{tail} of $y_t$ in the tree $\{y_j\}_{j=1}^\infty$
as the set
\[\{y_t,y_{2t},y_{2t+1},y_{4t}, y_{4t+1}, y_{4t+2},
y_{4t+3},\dots,y_{2^kt+1},\dots,y_{2^kt+2^k-1},\dots \}.\] (This
set can be informally described as the set of all ``branches''
which ``grow'' out of $y_t$.) We are going to use the observation
that \eqref{E:ineq1} and \eqref{E:ineq2} imply that
\begin{equation}\label{E:tail2j}
|y_m(x_{k+1,2r})|\ge 1-\ep
\end{equation}
for all $y_m$ in the tail of $y_{2j}$ and
\begin{equation}\label{E:tail2j+1}
|y_m(x_{k+1,2r})|\le \ep
\end{equation}
for all $y_m$ in the tail of $y_{2j+1}$.\medskip

Now we construct a bilipschitz embedding of $D_\omega$ into $X^*$.
Vertices of $D_0$ are mapped in the following way: one vertex is
mapped onto $0\in X^*$ and the other onto $y_1$ (the first element
of the tree). We continue in the following way: two new vertices
of $D_1$ are mapped onto $\frac{y_2}2$ and $\frac{y_3}2$,
respectively. The result will be a bilipschitz image of $D_1$
because by \eqref{E:x-sep2}, we have $\frac{1-2\ep}{1+\ep}\le
||y_2-y_3||\le ||y_2||+||y_3||\le 2$. We proceed in an obvious
way. To make this more clear we describe the next step.\medskip

In the obtained image of $D_1$ two edges correspond to
$\frac{y_2}2$ and two edges correspond to $\frac{y_3}2$. Here and
below we say that an edge \emph{corresponds to a vector} $z$ if
the difference between the images of the ends of the edge is $\pm
z$. In the diamond graphs edges are replaced by quadrilaterals. In
the images we replace the corresponding vectors by parallelograms.
When we say that a vector $z$ corresponding to an edge $uv$ is
\emph{replaced by a parallelogram with sides $x$ and $y$}, where
$x+y=z$, we mean the following. If $z=f(v)-f(u)$ and $u,a,v,b$ is
the quadrilateral which replaces the edge $uv$, then $f(a)=f(u)+x$
and $f(b)=f(u)+y$, so for the obtained mapping the edges $ua$ and
$bv$ correspond to $x$ and the edges $ub$ and $av$ correspond to
$y$.\medskip

Now we return to the construction of the embedding of $D_\omega$.
The vector $\frac{y_2}2$ is replaced by a parallelogram with sides
$\frac{y_4}4$ and $\frac{y_5}4$, and the mapping $f$ is extended
in the described in the previous paragraph way to each
quadrilateral which replaces edges corresponding to $\frac{y_2}2$.
In the next step the edge corresponding to $\frac{y_3}2$ is
replaced by a parallelogram with sides $\frac{y_6}4$ and
$\frac{y_7}4$, and the mapping $f$ is extended in the described in
the previous paragraph way to each quadrilateral which replaces
edges corresponding to $\frac{y_3}2$. So on.
\medskip

Since $||y_j||\le 1$ for each $j\in\mathbb{N}$, the constructed in
such a way embedding $f$ of $D_\omega$ into $X^*$ is
$1$-Lipschitz. It remains to show that it is bilipschitz. Here we
use some notions introduced by Johnson and Schechtman \cite{JS09}
in their study of embeddings of finite diamonds (see also
\cite[Section 9.3.2]{Ost13b}).\medskip

Namely, for any edge $uv$ in $D_k$ we let $S(u,v)$ be the union of
the following sequence of sets of vertices:

\begin{enumerate}
\item Vertices of the quadrilateral  which replaces the edge $uv$

\item Vertices of the quadrilaterals which replace the edges of
the quadrilateral from the previous item.

\item Vertices of the quadrilaterals which replace the edges of
the quadrilateral from the previous item.

\item So on.

\end{enumerate}

We call the set $D(u,v)$ a \emph{subdiamond} of $D_\omega$. It is
easy to see that for any two vertices $w,z\in D_\omega$ there is a
well-defined notion of the \emph{smallest subdiamond} containing
them. Let $D(u,v)$ be the smallest subdiamond containing $w$ and
$z$. Let $u,a,v,b$ be the quadrilateral which replaces the edge
$uv$ (when we form the next diamond). Then
\[D(u,v)=\left(D(u,a)\cup D(a,v)\right)\bigcup\left(D(u,b)\cup
D(b,v)\right).\]

We call the sets  $\left(D(u,a)\cup D(a,v)\right)$ and
$\left(D(u,b)\cup D(b,v)\right)$ the \emph{$a$-side} of the
subdiamond $D(u,v)$ and the \emph{$b$-side} of the subdiamond
$D(u,v)$, respectively. There are two possibilities:

\begin{itemize}

\item $w$ and $z$ are on the same side of $D(u,v)$.

\item $w$ and $z$ are on different sides of $D(u,v)$.

\end{itemize}

It is clear that in the first case we may assume that both $w$ and
$z$ are on the $a$-side of $D(u,v)$. Also since $D(u,v)$ is the
\emph{smallest} subdiamond containing $w$ and $z$, $w$ and $z$
cannot be both in $D(u,a)$ or both in $D(a,v)$. So we may assume
that $w\in D(u,a)$ and $z\in D(a,v)$.
\medskip

We have $f(a)-f(u)=\frac{y_{2j}}{2^k}$ and
$f(v)-f(a)=\frac{y_{2j+1}}{2^k}$ for some $j\in\mathbb{N}$, where
$k$ is determined by $2j=2^k+r$ with $0\le r<2^k$. (It can be that
$f(a)-f(u)=\frac{y_{2j+1}}{2^k}$ and
$f(v)-f(a)=\frac{y_{2j}}{2^k}$, but for our argument this does not
matter. However, to deal with the alternative case we would need
analogues of \eqref{E:tail2j} and \eqref{E:tail2j+1}
$x_{k+1,2r+1}$.)
\medskip

One can prove by induction the following two statements.

\begin{itemize}

\item[{\bf (a)}] The difference $f(a)-f(w)$ is a linear
combination with nonnegative coefficients of vectors contained in
the tail of $y_{2j}$ and the sum $\sigma_{a,w}$ of the
coefficients of this linear combination is equal to
$d_{D_\omega}(a,w)$.

\item[{\bf (b)}] The difference $f(z)-f(a)$ is a linear
combination with nonnegative coefficients of vectors contained in
the tail of $y_{2j+1}$ and the sum $\sigma_{z,a}$ of the
coefficients of this linear combination is equal to
$d_{D_\omega}(z,a)$.

\end{itemize}

To estimate $||f(z)-f(w)||$ from below we need to consider two
cases: $d_{D_\omega}(z,a)\le d_{D_\omega}(a,w)$ and
$d_{D_\omega}(z,a)\ge d_{D_\omega}(a,w)$. We consider the first
case only, the second case is similar. We have
\[\begin{split}||f(z)-f(w)||&\ge
\left|\frac{(f(z)-f(a))(x_{k+1,r})+(f(a)-f(w))(x_{k+1,r})}{1+\ep}\right|\\
&\stackrel{{\bf(a)}\&{\bf (b)}\&\eqref{E:tail2j}\&\eqref{E:tail2j+1}}{\ge}\frac{d_{D_\omega}(a,w)(1-\ep)-d_{D_\omega}(z,a)\ep}{1+\ep}\\
&\ge\frac{d_{D_\omega}(a,w)(1-2\ep)}{1+\ep}\ge
\frac{d_{D_\omega}(w,z)(1-2\ep)}{2(1+\ep)}.
\end{split}
\]

Now we consider the {\bf different sides} case. In this case there
are two subcases:

\begin{itemize}

\item[{\bf (A)}] Either we have both $d_{D_\omega}(w,u)\le
d_{D_\omega}(w,v)$ and $d_{D_\omega}(z,u)\le d_{D_\omega}(z,v)$,
or both $d_{D_\omega}(w,u)\ge d_{D_\omega}(w,v)$ and
$d_{D_\omega}(z,u)\ge d_{D_\omega}(z,v)$.

\item[{\bf (B)}] One of the vertices $w$ and $z$ closer to $u$ and
the other is closer to $v$. We may assume $d_{D_\omega}(w,u)<
d_{D_\omega}(w,v)$ and $d_{D_\omega}(z,u)> d_{D_\omega}(z,v)$.

\end{itemize}

In the subcase {\bf (A)} we use almost the same argument as above.
We may assume that both $w$ and $z$ are at least as close to $u$
as to $v$ and that  $d_{D_\omega}(w,u)\ge d_{D_\omega}(z,u)$. We
have $f(a)-f(u)=\frac{y_{2j}}{2^k}$ and
$f(b)-f(u)=\frac{y_{2j+1}}{2^k}$ for some $j\in\mathbb{N}$, where
$k$ is determined by $2j=2^k+r$ with $0\le r<2^k$. (It can happen
that $f(a)-f(u)=\frac{y_{2j+1}}{2^k}$ and
$f(b)-f(u)=\frac{y_{2j}}{2^k}$, but this case can be treated
similarly.) We also may assume that $w$ is on the
$a$-side.\medskip

One can prove by induction the following two statements.

\begin{itemize}

\item[{\bf (c)}] The difference $f(w)-f(u)$ is a linear
combination with nonnegative coefficients of vectors contained in
the tail of $y_{2j}$ and the sum $\sigma_{w,u}$ of the
coefficients of this linear combination is equal to
$d_{D_\omega}(w,u)$.

\item[{\bf (d)}] The difference $f(z)-f(u)$ is a linear
combination with nonnegative coefficients of vectors contained in
the tail of $y_{2j+1}$ and the sum $\sigma_{z,u}$ of the
coefficients of this linear combination is equal to
$d_{D_\omega}(z,u)$.

\end{itemize}

We have
\[\begin{split}||f(w)-f(z)||&\ge
\left|\frac{(f(w)-f(u))(x_{k+1,r})+(f(u)-f(z))(x_{k+1,r})}{1+\ep}\right|\\
&\stackrel{{\bf(c)}\&{\bf (d)}\&\eqref{E:tail2j}\&\eqref{E:tail2j+1}}{\ge}\frac{d_{D_\omega}(w,u)(1-\ep)-d_{D_\omega}(z,u)\ep}{1+\ep}\\
&\ge\frac{d_{D_\omega}(w,u)(1-2\ep)}{1+\ep}\ge
\frac{d_{D_\omega}(w,z)(1-2\ep)}{2(1+\ep)}.
\end{split}
\]

Now we consider subcase {\bf (B)}. We may assume that $w$ is
contained in the subdiamond $D(u,a)$ and $z$ is contained in the
subdiamond $D(b,v)$, where $u,a,v,b$ is the quadrilateral which
replaces the edge $uv$. We also may assume that
$f(a)-f(u)=\frac{y_{2j+1}}{2^k}$,
$f(v)-f(b)=\frac{y_{2j+1}}{2^k}$, and
$f(b)-f(u)=\frac{y_{2j}}{2^k}$ for some $j\in\mathbb{N}$, where
$k$ is determined by $2j=2^k+r$ with $0\le r<2^k$. One can prove
by induction the following statement.

\begin{itemize}

\item[{\bf (e)}] The differences $f(w)-f(u)$ and $f(z)-f(b)$ are
linear combinations with nonnegative coefficients of vectors
contained in the tail of $y_{2j+1}$ and the sums $\sigma_{w,u}$
and $\sigma_{z,b}$ of the coefficients of these linear
combinations are equal to $d_{D_\omega}(w,u)$ and
$d_{D_\omega}(w,u)$, respectively.

\end{itemize}

Observe that in the considered case the diameter of $D(u,v)$ is
equal to $\frac1{2^{k-1}}$, the diameters of $D(u,a)$ and $D(b,v)$
are equal to $\frac1{2^k}$. On the other hand, we have
\[\begin{split}||&f(z)-f(w)||\\&\ge
\left|\frac{(f(b)-f(u))(x_{k+1,r})+(f(z)-f(b))(x_{k+1,r})+(f(u)-f(w))(x_{k+1,r})}{1+\ep}\right|\\
&\stackrel{{\bf(e)}\&\eqref{E:tail2j}\&\eqref{E:tail2j+1}}{\ge}\frac{\frac1{2^k}(1-\ep)-d_{D_\omega}(z,b)\ep-d_{D_\omega}(u,w)\ep}{1+\ep}\\
&\ge\frac1{2^k}\cdot\frac{1-3\ep}{1+\ep}\\&\ge
\frac{d_{D_\omega}(w,z)(1-3\ep)}{2(1+\ep)}.
\end{split}
\]
This completes the proof in the case where $X^*$ is a nonseparable
dual of a separable Banach space.
\medskip

Now we consider the case where $X$ is nonseparable, but contains a
separable subspace $Y$ such that $Y^*$ is nonseparable. In this
case we apply Stegall's construction mentioned at the beginning of
the proof to $Y$ and denote the obtained collection of vectors in
$Y\subset X$ by $\{x_{n,i}\}_{n=0,~i=0}^{\infty\hskip0.45cm
2^n-1}$, and the obtained collection of nonempty weak$^*$ compact
convex subsets in $B_{Y^*}$ by
$\{V_{n,i}\}_{n=0,~i=0}^{\infty\hskip0.45cm 2^n-1}$. We let
$W_{n,i}$ be the set of all norm preserving extensions of
functionals of the set $V_{n,i}$ to the whole space $X$. It is
clear that $W_{n,i}$ are nonempty weak$^*$ compact convex subsets
in $B_{X^*}$, and that they satisfy the conditions
\eqref{E:dyadic} and \eqref{E:disjoint}. It is also easy to verify
that if we construct $\Delta$ in the same way as before, the
condition \eqref{E:Stegall} also continues to hold. Therefore
everything in the proof can be done in the same way as in the case
where $X$ is separable.
\end{proof}

\section{A submetric test space for the infinite tree
property, proof of Theorem
\ref{T:SubInfTree}}\label{S:SubMetrJames}

\begin{proof}[Proof of Theorem
\ref{T:SubInfTree}] On one hand, if we analyze the proof of
Theorem \ref{T:BulgDiamond} in \cite{Ost11}, we see that we used
the bilipschitz condition only for pairs of points which are in
the same quadrilateral.
\medskip

On the other hand, let $\{x_i\}_{i=1}^\infty$ be a bounded
$\delta$-tree in a Banach space $Z$. First we shift the tree in
order to achieve the situation in which it is bounded both from
below (in the sense that $\exists c>0~\forall
i\in\mathbb{N}~||x_i||\ge c$) and from above. Now we construct the
image of the submetric space $(D_\omega,S_\omega)$ in the
following way (the construction of \cite{Ost11} backward):

We map vertices of $D_0$ onto $0$ and $x_1$, respectively. We map
the new vertices $a$ and $b$ of the quadrangle which replaces
$D_0$ to $x_2/2$ and $x_3/2$, respectively. It is clear that
because $x_1$ is the sum of $x_2/2$ and $x_3/2$, that all edges of
$D_1$ correspond to one of them (correspond in the sense that they
are differences between the vectors corresponding to the
vertices).\medskip

We continue as follows. Since $x_2/2=x_4/4+x_5/4$ we let the edges
of the second level in quadruples having $x_2/2$ as the difference
between the top and the bottom to be: (the bottom)$+x_4/4$ and
(the bottom)$+x_5/4$. We continue in an obvious way.\medskip

It is easy to verify that the fact that the tree is a bounded
$\delta$-tree and norms of its elements are bounded away from zero
implies that we get a partial bilipschitz embedding of the
submetric space $(D_\omega,S_\omega)$ into the Banach space $Z$.
(To visualize the proof it is worthwhile to sketch a $D_2$ and to
label edges by differences between the vectors corresponding to
their end vertices.) \end{proof}

\section{Classes of metric spaces which do not admit bilipschitz embeddings into spaces with the Radon-Nikod\'ym
property}\label{S:MetrRNP}

The conditions implying non-embeddability, which we are going to
present in this section are of the type:  Any bilipschitz image of
a metric space $X$ in a Banach space $Y$ contains a set which can
be used to form a bounded divergent martingale (the values of the
martingale are multiples of differences between images of certain
points of the metric space). We use the following result of
Chatterji \cite{Cha68} (see also \cite{BL00}, \cite{Bou83},
\cite{DU77}, and \cite{Pis11}): A Banach space $Y$ has the RNP if
and only if each bounded $Y$-valued martingale converges.

\subsection{Spaces with thick families of geodesics between some pairs of points}

The purpose of this section is to prove the following
generalization of the result of \cite{Ost11}:

\begin{theorem}\label{T:MetrNonRNP} Let $(X,d)$ be a metric space satisfying the following two conditions:

\begin{enumerate}

\item[{\bf (1)}] There are two points $u$ and $v$ in $X$ and
infinitely many marked geodesics between them in the completion
$\widetilde X$ of $X$ such that the following condition is
satisfied. {\rm (Not all of the geodesics between $u$ and $v$ have
to be marked.)}

\item[{\bf (2)}]\label{I:Thick} For any two points $u_0$ and $v_0$
on a marked $uv$-geodesic there are points \[w_0=u_0, w_1, \dots,
w_{n-1}, w_n=v_0\] which also lie on some marked $uv$-geodesic,
and their order on the geodesic coincides with the order in which
they are listed; and there are points $\{z_i,\widetilde
z_i\}_{i=1}^n$ such that

\begin{enumerate}

\item The points $w_0=u_0, z_1, w_1, z_2, \dots, w_{n-1}, z_n,
w_n=v_0$ lie on a marked $uv$-geodesic and are listed in their
order on the geodesic.

\item The points $w_0=u_0, \widetilde z_1, w_1, \widetilde z_2,
\dots, w_{n-1}, \widetilde z_n, w_n=v_0$ lie on a marked
$uv$-geodesic and are listed in their order on the geodesic.

{\rm These geodesic have to be different because of the next two
conditions.}

\item\label{I:SameLevel} $d(w_i,z_i)=d(w_i,\widetilde z_i)$ and
$d(w_{i-1},z_i)=d(w_{i-1},\widetilde z_i)$.

\item \begin{equation}\label{E:Width}\sum_{i=1}^n d(z_i,\widetilde
z_i)\ge c d(u_0,v_0),\end{equation} where $c$ depends on $X$ but
not on the choice of $u_0$ and $v_0$.

\end{enumerate}
\end{enumerate}

Then the metric space $(X,d)$ does not admit a bilipschitz
embedding into a Banach space with the Radon-Nikod\'ym property.
\end{theorem}

\begin{proof}
We assume that $(X,d)$ admits a bilipschitz embedding $f:X\to Y$
into a Banach space $Y$ and show that there exists a bounded
divergent martingale  $\{M_i\}_{i=0}^\infty$ on $(0,1]$ with
values in $Y$. We assume that
\begin{equation}\label{E:Bilip}
\ell d(x,y)\le ||f(x)-f(y)||_Y\le d(x,y)
\end{equation}
for some $\ell>0$. We assume that $d(u,v)=1$ (dividing all
distances in $X$ by $d(u,v)$, if necessary).
\medskip

Each function in the martingale $\{M_i\}_{i=0}^\infty$ will be
obtained in the following way. We consider some finite sequence
$V=\{v_i\}_{i=0}^m$ of points on a (not necessarily marked)
$uv$-geodesic, satisfying $v_0=u$, $v_m=v$ and $d(u,v_{k+1})\ge
d(u,v_k)$. We define $M_V$ as the function on $(0,1]$ whose value
on the interval $(d(u,v_k),d(u,v_{k+1})]$ is equal to
\[\frac{f(v_{k+1})-f(v_k)}{d(v_k,v_{k+1})}.\]

It is clear that \eqref{E:Bilip} implies that $||M_V(t)||\le 1$
for any collection $V$ and any $t\in(0,1]$. Also it is clear that
an infinite collection of such functions
$\{M_{V(k)}\}_{k=0}^\infty$ forms a martingale if for each
$k\in\mathbb{N}$ the sequence $V(k)$ contains $V(k-1)$ as a
subsequence . So it remains to to find such increasing collection
of sequences $\{V(k)\}_{k=0}^\infty$ for which the martingale
$\{M_{V(k)}\}_{k=0}^\infty$  diverges. We denote $M_{V(k)}$ by
$M_k$.\medskip

We let $V(0)=\{u,v\}$ and so $M_0$ is a constant function on
$(0,1]$ taking value $f(v)-f(u)$. In the next step we apply the
condition {\bf (2)} to $v_0=v$ and $u_0=u$ and find the
corresponding sequences $\{w_i\}_{i=0}^n$ and $\{z_i,\widetilde
z_i\}_{i=1}^n$. We let $V(1)=\{w_i\}_{i=0}^n$. Observe, that in
this step we cannot claim any nontrivial estimates for
$||M_1-M_0||_{L_1(Y)}$ from below because we have not made any
nontrivial assumptions on this step of the construction. Lower
estimates for martingale differences in our argument are obtained
only for differences of the form $||M_{2k}-M_{2k-1}||_{L_1(Y)}$.
\medskip

We choose $V(2)$ to be of the form
\begin{equation}\label{E:MartSeq} w_0, z'_1, w_1, z'_2, w_n,\dots,
z'_n,w_n,\end{equation} where each $z'_i$ is either $z_i$ or
$\widetilde z_i$ depending on the behavior of the mapping $f$. We
describe this dependence below. Observe that by condition (c) of
Theorem \ref{T:MetrNonRNP}, the corresponding partition of the
interval $(0,1]$ does not depend on whether we choose $z_i$ or
$\widetilde z_i$.\medskip

To make the choice of $z_i'$ we consider the quadrilateral
$w_{i-1}, z_i, w_i, \widetilde z_i$. Inequality \eqref{E:Bilip}
implies $||f(z_i)-f(\widetilde z_i)||\ge\ell d(z_i,\widetilde
z_i)$. Consider two pairs of vectors corresponding to two
different choices of $z'_i$:
\medskip

\noindent{\bf Pair 1:} $f(w_i)-f(z_i)$, $f(z_i)-f(w_{i-1})$.\qquad
{\bf Pair 2:} $f(w_i)-f(\widetilde z_i)$, $f(\widetilde
z_i)-f(w_{i-1})$.
\medskip

The inequality $||f(z_i)-f(\widetilde z_i)||\ge\ell
d(z_i,\widetilde z_i)$ implies that at least one of the following
is true
\begin{equation}\label{E:z}\begin{split}\left\|\frac{f(w_i)-f(z_i)}{d(w_i,z_i)}-\frac{f(z_i)-f(w_{i-1})}{d(z_i,w_{i-1})}\right\|
\ge\frac{\ell}2\, d(z_i,\widetilde
z_i)\left(\frac1{d(w_i,z_i)}+\frac1{d(z_i,w_{i-1})}\right)\end{split}
\end{equation}
or
\begin{equation}\label{E:zTilde}\begin{split}\left\|\frac{f(w_i)-f(\widetilde
z_i)}{d(w_i,\widetilde z_i)}-\frac{f(\widetilde
z_i)-f(w_{i-1})}{d(\widetilde
z_i,w_{i-1})}\right\|\ge\frac{\ell}2\, d(z_i,\widetilde
z_i)\left(\frac1{d(w_i,\widetilde z_i)}+\frac1{d(\widetilde
z_i,w_{i-1})}\right)\end{split}
\end{equation}
We pick $z'_i$ to be $z_i$ if the left-hand side of \eqref{E:z} is
larger than the left-hand side of \eqref{E:zTilde}, and pick
$z'_i=\widetilde z_i$ otherwise.\medskip

Let us estimate $||M_2-M_1||_1$. First we estimate the part of
this difference corresponding to the interval
$\left({d(w_0,w_{i-1})},{d(w_0,w_i)}\right]$. Since the
restriction of $M_2$ to the interval
$\left({d(w_0,w_{i-1})},{d(w_0,w_i)}\right]$ is a two-valued
function, and $M_1$ is constant on the interval, the integral
\begin{equation}\label{E:IntOneInt}\int_{{d(w_0,w_{i-1})}}^{{d(w_0,w_{i})}}||M_2-M_1||dt\end{equation}
can be estimated from below in the following way. Denote the value
of $M_2$ on the first part of the interval by $x$, the value on
the second by $y$, the value of $M_1$ on the whole interval by
$z$, the length of the first interval by $A$ and of the second by
$B$. We have: the desired integral is equal to $A||x-z||+B||y-z||$
and therefore can be estimated in the following way:
\[\begin{split}A||x-z||+B||y-z||&\ge\max\{||x-z||,||y-z||\}\cdot\min\{A,B\}\\&\ge\frac12||x-y||\min\{A,B\}.\end{split}\]

Therefore, assuming without loss of generality that the left-hand
side of \eqref{E:z} is larger than the left-hand side of
\eqref{E:zTilde}, the integral in \eqref{E:IntOneInt} can be
estimated from below by
\[\begin{split}&\frac12\left\|\frac{(f(w_i)-f(z_i))}{d(w_i,z_i)}-\frac{(f(z_i)-f(w_{i-1}))}{d(z_i,w_{i-1})}\right\|
\cdot\min\left\{{d(w_i,z_i)},{d(z_i,w_{i-1})}\right\}
\\&\qquad\ge\frac14\,\ell\, d(z_i,\widetilde z_i)
\left(\frac1{d(w_i,z_i)}+\frac1{d(z_i,w_{i-1})}\right)\cdot\min\left\{{d(w_i,z_i)},{d(z_i,w_{i-1})}\right\}
\\&\qquad\ge\frac14\,\ell\,{d(z_i,\widetilde z_i)}.\end{split}
\]
Summing over all intervals and using the condition
\eqref{E:Width}, we get $||M_2-M_1||\ge \frac14\ell c\,d(u,v)$.
\medskip

Now we apply the condition {\bf (2)} for each pair of consecutive
points in the sequence
\begin{equation}\label{E:listBefore(2)}w_0,
z'_1, w_1, z'_2, w_2,\dots, z'_n,w_n,\end{equation} where each
$z'_i$ is either $z_i$ or $\widetilde z_i$ depending on the choice
made above. We list all of the obtained $w$-points in one list
\begin{equation}\label{E:NewListW} w^1_0, w^1_1, w^1_2,\dots, w^1_{n(1)}, \end{equation} and the
obtained $z$-points as two collections:
\[z^1_1, z^1_2,\dots, z^1_{n(1)}, \] and
\[\widetilde z^1_1, \widetilde z^1_2,\dots, \widetilde
z^1_{n(1)}.\] Observe that the whole sequence \eqref{E:NewListW}
does not have to be on the same marked geodesic, only pieces which
correspond to pairs of consecutive points in the list
\eqref{E:listBefore(2)} are on marked geodesics. However this
implies that any list of the form
\[w^1_0, z'^1_1, w^1_1, z'^1_2, w^1_2,\dots, z'^1_{n(1)}, w^1_{n(1)},\]
where each $z'^1_1$ is either $z^1_1$ or $\widetilde z^1_1$ is on
some (not necessarily marked) $uv$-geodesic. Because of this we
can proceed in the same way as before in our construction of the
$Y$-valued functions $M_3$ and $M_4$. Again, we have no estimate
for $||M_2-M_3||$, but we get the same estimate for $||M_3-M_4||$.
We proceed in an obvious way. As a result we get a bounded
divergent martingale.
\end{proof}

\subsection{Application to the second Laakso
space}\label{S:SecondLaakso}

Our next goal is to show that the second Laakso space
\cite{Laa00}, whose construction we present following Cheeger and
Kleiner \cite[Example 1.4]{CK13} satisfies the conditions of
Theorem \ref{T:MetrNonRNP}, and thus to get a different proof of
the result of Cheeger and Kleiner \cite[Corollary 1.7]{CK09}
stating that this space does not admit bilipschitz embeddings into
Banach spaces with the \RNP.

\begin{example}[{\bf Second Laakso space}]
\label{Ex:LaaksoAhlforsRegularPI} We construct this space as an
inductive limit of graph thickenings, that is, graphs in which
edges are isometric to line segments of the corresponding lengths,
and elements of the edges  (not only ends) are elements of the
metric spaces. We start with a space $X_0$ which has two vertices
and an edge of length $1$ joining them. Given $X_{i-1}$ we
construct $X_i$ in two steps.
\medskip

\noindent{\bf Step 1.} We replace each edge in $X_i$ by a path of
the same length consisting of three edges of equal length. In
other words we trisect each edge in the sense that we insert new
vertices after each third of it. We denoted the set of all new
vertices by $N_i$ and the obtained graph (topologically the same
as $X_{i}$, but with many new vertices) by $X'_i$.
\medskip

\noindent{\bf Step 2.} We consider two copies of $X'_i$ and paste
them at the respective vertices of $N_i$. We introduce $X_{i+1}$
as the obtained graph thickening with its shortest path distance.
More formally, we let $X_{i+1}$ be the set of equivalence classes
of $X_{i}'\times \{0,1\}$ with respect to the equivalence given by
$(v,0)\sim (v,1)$ for all $v\in N_i$ (all other equivalence
classes are one-element sets).
\medskip

We consider $X_{i+1}$ as a metric space with its shortest path
distance. Observe that there are natural isometric embeddings
$X_i\to X_{i+1}$ (we identify $X_i$ with $X_i\times\{0\}$). We let
$X_\omega$ be the union of $X_i$, with the metric
$d_{X_\omega}(u,v)$ defined as $d_{X_\omega}(u,v)=d_{X_i}(u,v)$,
where $i\in\mathbb{N}$ is large enough so that $u,v\in X_i$.
\end{example}

\begin{proposition}\label{P:XomegaThick} The space $X_\omega$ satisfies all conditions
of Theorem \ref{T:MetrNonRNP} if we pick $u$ and $v$ to be the
vertices of $X_0$ and consider all geodesics joining them in any
of $X_i$ as marked.
\end{proposition}

\begin{proof} Let $u_0$ and $v_0$ be two points on a (marked)
$uv$-geodesic. Let $j\in \mathbb{N}$ be such that $u_0,v_0\in
X_j$. We trisect all edges of $X_j$ and let $a^{1}_0, \dots,
a^{1}_{m(1)}$ be the vertices of $N_j$ on one of the geodesics $S$
joining $u_0$ and $v_0$, in the order, in which we meet them
travelling from $u_0$ to $v_0$. We denote $d_{X_\omega}$ by $d$.
If $d(a^{1}_0,a^{1}_{m(1)})\ge\frac12 d(u_0,v_0)$ (the number
$\frac12$ can be replaced by any other number in (0,1), this will
affect only the constant $c$ in \eqref{E:Width}), we let $n=m(1)$
and the vertices $w_0,w_1,\dots,w_n$ (see condition {\bf (2)} in
Theorem \ref{T:MetrNonRNP}) be given by $w_0=u_0, w_1=a^{1}_1,
\dots, w_{n-1}=a^{1}_{n-1}, w_n=v_0$.
\medskip

If $d(a^{1}_0,a^{1}_{m(1)})<\frac12 d(u_0,v_0)$, we repeat the
trisection (now we trisect edges of $X_{j+1}$) and let $a^{2}_0,
\dots, a^{2}_{m(2)}$ be vertices of the new trisection on $S$. If
$d(a^{2}_0,a^{2}_{m(2)})\ge\frac12 d(u_0,v_0)$, we let $n=m(2)$
and the vertices $w_0,w_1,\dots,w_n$ (see condition {\bf (2)} in
Theorem \ref{T:MetrNonRNP}) be given by $w_0=u_0, w_1=a^{2}_1,
\dots, w_{n-1}=a^{2}_{n-1}, w_n=v_0$.
\medskip

If not, we repeat the trisection, so on. It is clear that there is
$k\in\mathbb{N}$ such that after $k$ trisections  the condition
$d(a^{k}_0,a^{k}_{n(k)})\ge\frac12 d(u_0,v_0)$ is satisfied. At
that point we define $n=n(k)$ and $w_0,\dots,w_n$, by $w_0=u_0,
w_1=a^{k}_1, \dots, w_{n-1}=a^{k}_{n-1}, w_n=v_0$.
\medskip

Now we introduce $z_1,\dots,z_n$ and $\widetilde
z_1,\dots,\widetilde z_n$. Observe that $a^{k}_0, \dots,
a^{k}_{n}$ are vertices in $N_{j+k-1}$. We let $z_i$ $i=1,\dots,n$
be the midpoint on $S$ between $a^{k}_{i-1}$ and $a^{k}_i$. Now we
recall that when we create $X_{j+k}$ we consider two copies of
$X_{j+k-1}$ and paste them at $N_{j+k-1}$. After that (in the
notation introduced above) points $z_i$ are identified with pairs
$(z_i,0)$. We introduce $\widetilde z_i$ as $(z_i,1)$. It remains
to verify that condition \eqref{E:Width} is satisfied.

It is easy to see that $d(z_i,\widetilde z_i)=d(a^{k}_{i-1},
a^{k}_i)$. Therefore
\[\sum_{i=1}^nd(z_i,\widetilde
z_i)=d(a^k_0,a^k_n)\ge\frac12 d(u_0,v_0).\qedhere\]
\end{proof}

Combining Proposition \ref{P:XomegaThick} with  Theorem
\ref{T:MetrNonRNP} we get the following result of Cheeger and
Kleiner \cite[Corollary 1.7]{CK09}.

\begin{corollary} $X_\omega$ does not admit a bilipschitz
embedding into a Banach space with the \RNP.
\end{corollary}

\subsection{A wider class of metric spaces which do not admit bilipschitz embeddings into \RNP\ spaces}

The purpose of this section is to prove  an ``isomorphic'' version
of Theorem \ref{T:MetrNonRNP}. We need the following notion.

\begin{definition}\label{D:Cuv}
Let $C\in[1,\infty)$ and $u,v$ are two points in a metric space. A
\emph{$C$-geodesic between $u$ and $v$}, also called a
\emph{$Cuv$-geodesic}, is a finite sequence
\[u_0=u, u_1, u_2,\dots,u_n=v\]
of points satisfying
\begin{equation}\label{E:CGeo}\sum_{i=1}^nd(u_i,u_{i-1})\le C\cdot d(u,v).\end{equation} A
$Cuv$-geodesic $v_0,\dots,v_m$ is called an \emph{extension} of a
$Cuv$-geodesic $u_0,\dots,u_n$ if $m\ge n$ and  $u_0,\dots,u_n$ is
a subsequence of $v_0,\dots,v_m$,

\end{definition}

\begin{remark} This notion of a $C$-geodesic is in a certain sense too wide because the distances $d(u_i,u_{i+1})$ are allowed to be ``large''. In
particular, a $Cuv$-geodesic can be trivial: the sequence
$u_0,u_1$ with $u_0=u$ and $u_1=v$ is a $Cuv$-geodesic for each
$C\ge 1$.
\end{remark}

We are going to construct martingales corresponding to bilipschitz
embeddings of metric spaces having such geodesics. In this
connection we introduce the following terminology.

\begin{definition}\label{D:PartRef} Let $u_0=u, u_1, u_2,\dots,u_n=v$ be a
$Cuv$-geodesic. We introduce the numbers
\begin{equation}\label{E:Intervals}
\alpha_k:=\frac{\sum_{i=1}^k d(u_i,u_{i-1})}{\sum_{i=1}^n
d(u_i,u_{i-1})}, k=1,\dots,n, \quad \alpha_0=0. \end{equation}
These numbers induce a partition of the interval $[0,1]$ (for our
purposes it does not matter how we include the points
$\{\alpha_i\}_{i=0}^n$ into the intervals of the partition). We
call this partition the \emph{partition corresponding to the
$Cuv$-geodesic}  $u_0, u_1, u_2,\dots,u_n$.

If we consider a sequence of $Cuv$-geodesics in which each next
$Cuv$-geodesic is an extension of the previous one, we form the
following increasing sequence of partitions of $[0,1]$. The first
partition is the partition corresponding to the first
$Cuv$-geodesic. The second partition is a refinement of the first
partition obtained in the following way. Let $\alpha_i$ and
$\alpha_{i+1}$ be two consecutive points in the fist partition
$w_i$ and $w_{i+1}$ be the corresponding points in the geodesic.
Let $z_j=w_i,\dots,z_k=w_{i+1}$ be the corresponding points in the
extended $Cuv$-geodesic. The interval $[\alpha_i,\alpha_{i+1}]$ in
the refinement is partitioned by the points
\[\begin{split}\alpha_i,\alpha_i+\frac{(\alpha_{i+1}-\alpha_i)d(z_j,z_{j+1})}{\sum_{t=j}^{k-1}d(z_t,z_{t+1})},
&\alpha_i+\frac{(\alpha_{i+1}-\alpha_i)\sum_{t=j}^{j+1}d(z_t,z_{t+1})}{\sum_{t=j}^{k-1}d(z_t,z_{t+1})},
\dots,\\
&\alpha_i+\frac{(\alpha_{i+1}-\alpha_i)\sum_{t=j}^{k-2}d(z_t,z_{t+1})}{\sum_{t=j}^{k-1}d(z_t,z_{t+1})},
\alpha_{i+1}.\end{split}
\]
The same is done for each interval of the partition. This
procedure of refinement is repeated for each further extension of
$Cuv$-geodesics.
\end{definition}

\begin{theorem}\label{T:IsoMNonRNP} Let $(X,d)$ be a metric space for which there are
two points $u$ and $v$ in $X$ and a family of marked
$Cuv$-geodesics such that the following conditions are satisfied:

\begin{enumerate}

\item\label{I:ManyExt} Any marked $Cuv$-geodesic
\[w_0=u, w_1, \dots, w_{n-1}, w_n=v\] has two different marked extensions
\begin{equation}\label{E:ExtZ} z_0=u, z_1,\dots, z_m=v\end{equation} and
\begin{equation}\label{E:ExtZTilde}\widetilde z_0=u, \widetilde z_1,\dots, \widetilde z_m=v.
\end{equation}

These extensions are such that for any sequence $z'_0,
z'_1,\dots,z'_m$ in which each $z'_i$ is either $z_i$ or
$\widetilde z_i$ is also a marked $Cuv$-geodesic. Furthermore, the
extensions \eqref{E:ExtZ} and \eqref{E:ExtZTilde} satisfy the
following conditions:

\begin{enumerate}

\item They have some more common points in addition to $w_0, w_1,
\dots, w_{n-1}, w_n$, and all common points $\{z_i\}_{i\in C}$
have the same indices in both sequences, and form a marked
$Cuv$-geodesic.


\item Have some pairs of distinct points $\{z_i,\widetilde
z_i\}_{i\in D}$ which satisfy

\begin{itemize}

\item Each pair of distinct points is between two pairs of common
points.\medskip

\item
\begin{equation}\label{E:EqualProp}\displaystyle{\frac{d(z_i,z_{i-1})}{d(z_{i},z_{i+1})}=
\frac{d(\widetilde z_i,z_{i-1})}{d(\widetilde
z_{i},z_{i+1})}}.\end{equation}

\item \begin{equation}\label{E:DistGeod}\sum_{i\in D}
d(z_i,\widetilde z_i)\ge c d(u,v),\end{equation} where $c$ does
not depend on the choice of a marked $Cuv$-geodesic \[w_0=u, w_1,
\dots, w_{n-1}, w_n=v.\]

\end{itemize}

\end{enumerate}

\item\label{I:NonAccum} The set of all marked $C$-geodesics
satisfies the following condition of ``non-accumulation of
distortion of partitions'': there exists a constant
$B\in[1,\infty)$ such that if we consider a collection of marked
$Cuv$-geodesics with each next being an extension of the previous
one and consider the corresponding nested family of partitions,
each next refining the previous one \emph{(according to the
construction of Definition \ref{D:PartRef})}, then the final
partition will be $B$-equivalent to the one corresponding to the
last geodesic in the sequence. By $B$-equivalence we mean that
\[\frac1B\le\frac{\alpha_{i+1}-\alpha_i}{\beta_{i+1}-\beta_i}\le
B.\]

\end{enumerate}

Then the metric space $(X,d)$ does not admit bilipschitz
embeddings into Banach spaces with the Radon-Nikod\'ym property.
\end{theorem}

\begin{proof}
We assume that $(X,d)$ admits a bilipschitz embedding $f:X\to Y$
into a Banach space $Y$ and show that there exists a bounded
divergent martingale on $[0,1]$ with values in $Y$. We assume that
\begin{equation}\label{E:Bilip2}
\ell d(x,y)\le ||f(x)-f(y)||_Y\le d(x,y)
\end{equation}
for some $\ell>0$.\medskip

We are going to construct a bounded divergent $Y$-valued
martingale. We start our construction of the martingale by picking
any marked $Cuv$-geodesic $w_0=u, w_1, \dots, w_{n-1}, w_n=v$. Let
$\alpha_0=0,\alpha_1,\dots,\alpha_n=1$ be the ends of the
corresponding partition of $[0,1]$. We introduce the first
function of the martingale, $M_0:[0,1]\to Y$, by
\[M_0(t)=\frac{f(w_{i+1})-f(w_{i})}{\alpha_{i+1}-\alpha_i}\quad \hbox{for}\quad t\in\left[\alpha_{i}, \alpha_{i+1}\right]\]
All functions of our martingale will be of this type for
partitions obtained by sequences of refinements done according to
Definition \ref{D:PartRef}. We observe that condition
\ref{I:NonAccum} of Theorem \ref{T:IsoMNonRNP} in combination with
\eqref{E:Intervals} and \eqref{E:Bilip2} implies the boundedness
of the martingale. In fact, in the case where
$[\alpha_i,\alpha_{i+1}]$ form a partition corresponding to a
$Cuv$-geodesic $w_0=u, w_1, \dots, w_{n-1}, w_n=v$ we have
\[\left\|\frac{f(w_{i+1})-f(w_{i})}{\alpha_{i+1}-\alpha_i}\right\|=\left\|\frac{(f(w_{i+1})-f(w_{i}))\sum_{i=1}^n
d(w_i,w_{i-1})}{d(w_i,w_{i+1})}\right\|\le Cd(u,v),
\]
where we use \eqref{E:Bilip2} and the definition of a
$C$-geodesic. In the further steps the boundedness of the integral
follows from condition \ref{I:NonAccum} of Theorem
\ref{T:IsoMNonRNP}.
\medskip

In the next step we apply condition \ref{I:ManyExt} of Theorem
\ref{T:IsoMNonRNP} to the marked $Cuv$-geodesic $w_0=u, w_1,
\dots, w_{n-1}, w_n=v$. We get two marked $Cuv$-geodesics
\[z_0=u, z_1,\dots, z_m=v\]
and
\[\widetilde z_0=u, \widetilde z_1,\dots, \widetilde z_m=v.
\]

Let $\{z_i\}_{i\in C}=\{\widetilde z_i\}_{i\in C}$ be the set of
their common points, which also form a marked $Cuv$-geodesic, we
denote it $\{y_i\}$. We consider the refinement of the partition
given by $\alpha_i$ corresponding to the extension $\{y_i\}$ of
the geodesic $\{w_i\}$. Let $\{\beta_i\}$ be the points of
division of the corresponding partition of $[0,1]$. Then we divide
the subintervals containing distinct points in the corresponding
proportion. The condition \eqref{E:EqualProp} implies that we get
the equal divisions for both geodesics. We denote the obtained
division points $\{\gamma_i\}_{i=0}^m$, $\gamma_0=0$,
$\gamma_m=1$.
\medskip

Now we define two next functions of the martingale:
\[M_1(t)=\frac{f(y_{i+1})-f(y_{i})}{\beta_{i+1}-\beta_i}\quad \hbox{for}\quad t\in\left[\beta_{i}, \beta_{i+1}\right].
\]
It is clear that $M_0$ is the conditional expectation of $M_1$
with respect to the $\sigma$-algebra generated by $\alpha$.
\medskip

In the next step we do the same for the points $z_i$ and
$\widetilde z_i$ and numbers $\{\gamma_i\}$. The only difference
is that now for each $i\in D$ we pick either $z_i$ or $\widetilde
z_i$. We definitely get a martingale, no matter whether we pick
$z_i$ or $\widetilde z_i$ in each of the steps. The point of the
choice is to make the (eventually constructed) martingale
divergent. To achieve this goal in we observe that for each $i\in
D$ at least one of the following inequalities holds:

\begin{equation}\label{E:zIso}\begin{split}\left\|\frac{f(z_i)-f(z_{i-1})}{\gamma_i-\gamma_{i-1}}-\frac{f(z_{i+1})-f(z_{i})}{\gamma_{i+1}-\gamma_i}\right\|
\ge\frac{\ell}2\, d(z_i,\widetilde
z_i)\left(\frac{1}{\gamma_i-\gamma_{i-1}}+\frac{1}{\gamma_{i+1}-\gamma_i}\right)\end{split}
\end{equation}
or
\begin{equation}\label{E:zTildeIso}\begin{split}\left\|\frac{f(\widetilde z_i)-f(z_{i-1})}{\gamma_i-\gamma_{i-1}}-\frac{f(z_{i+1})-f(\widetilde z_{i})}{\gamma_{i+1}-\gamma_i}\right\|
\ge\frac{\ell}2\, d(z_i,\widetilde
z_i)\left(\frac{1}{\gamma_i-\gamma_{i-1}}+\frac{1}{\gamma_{i+1}-\gamma_i}\right)\end{split}
\end{equation}

In fact, otherwise we get
\[
||f(z_i)-f(\widetilde
z_i)||\left(\frac{1}{\gamma_i-\gamma_{i-1}}+\frac{1}{\gamma_{i+1}-\gamma_i}\right)<
{\ell}d(z_i,\widetilde
z_i)\left(\frac{1}{\gamma_i-\gamma_{i-1}}+\frac{1}{\gamma_{i+1}-\gamma_i}\right),
\]
we get a contradiction with \eqref{E:Bilip2}. This choice of the
$Cuv$-geodesic allows us to get a lower estimate for
$\int_0^1||M_2(t)-M_1(t)||dt$. We start by estimating the part of
this difference corresponding to the interval
$[\gamma_{i-1},\gamma_{i+1}]$. Since the restriction of $M_2$ to
this interval is a two-valued function, and $M_1$ is constant on
the interval, integral
\begin{equation}\label{E:IntOneInt2}\int_{\gamma_{i-1}}^{\gamma_{i+1}}||M_2-M_1||dt\end{equation}
can be estimated from below in the following way. Denote the value
of $M_2$ on the first part of the interval by $x$, the value on
the second by $y$, the value of $M_1$ on the whole interval by
$z$, the length of the first interval by $A_1$ and of the second
by $A_2$. We have: the desired integral is equal to
$A_1||x-z||+A_2||y-z||$ and therefore can be estimated in the
following way:
\[\begin{split}A_1||x-z||+A_2||y-z||&\ge\max\{||x-z||,||y-z||\}\cdot\min\{A_1,A_2\}\\&\ge\frac12||x-y||\min\{A_1,A_2\}.\end{split}\]

Therefore, assuming without loss of generality that the left-hand
side of \eqref{E:zIso} is larger than the left-hand side of
\eqref{E:zTildeIso}, the integral in \eqref{E:IntOneInt2} can be
estimated from below by
\[\begin{split}\frac12\cdot\frac{\ell}2\, d(z_i,\widetilde
z_i)\left(\frac{1}{\gamma_i-\gamma_{i-1}}+\frac{1}{\gamma_{i+1}-\gamma_i}\right)
\min\left\{\gamma_i-\gamma_{i-1},\gamma_{i+1}-\gamma_i\right\}
\ge\frac{\ell}4\,{d(z_i,\widetilde z_i)}.\end{split}
\]
Summing over all intervals and using the condition
\eqref{E:DistGeod}, we get
\begin{equation}\label{E:Diff12}||M_2-M_1||\ge \frac{\ell}4\,
c\,d(u,v).\end{equation}

We continue in an obvious way. It is clear that similarly to
\eqref{E:Diff12} we get \[||M_{2n}-M_{2n-1}||\ge \frac{\ell}4\,
c\,d(u,v)\] for each $n\in\mathbb{N}$.
\end{proof}

\subsection{Metric spaces satisfying the conditions of Theorem
\ref{T:MetrNonRNP} or \ref{T:IsoMNonRNP} are not test spaces for
the RNP}

\begin{theorem}\label{T:NotTestSp}
For each metric space $X$ satisfying the conditions of Theorem
\ref{T:MetrNonRNP} or \ref{T:IsoMNonRNP} there exists a subspace
of $L_1(0,1)$ which does not have the \RNP\ and does not admit a
bilipschitz embedding of $X$. Hence none of such metric spaces is
a test space for the \RNP.
\end{theorem}

Proof of this theorem is based on the following result of Bourgain
and Rosenthal \cite{BR80} (see also expositions of this result in
\cite[Section 4 of Chapter 5]{BL00} and \cite{Bou83}).

\begin{theorem}[Bourgain-Rosenthal]\label{T:BourgainRos} Let $\{k_n\}_{n=1}^\infty$ be any increasing
sequence of positive integers. There exists a subspace $E$ of
$L_1(0,1)$ which does not have the Radon-Nikod\'ym property, but
is such that each $E$-valued martingale $\{f_n\}_{n=1}^\infty$ on
$[0,1]$ adapted to a sequence  $\{\mathcal{B}_n\}_{n=1}^\infty$ of
finite $\sigma$-algebras satisfying $|\mathcal{B}_n|\le k_n$ and
bounded in the sense that $\sup_n||f_n||_\infty\le 1$ satisfies
$\liminf_{n\to\infty}||f_{n+1}-f_n||_1=0$.
\end{theorem}

This theorem is almost mentioned in \cite[top of page~55]{BR80}
and \cite[Remark on page 121]{BL00}. However the statements are
somewhat different. For this reason we describe the modifications
which should be done in the proof of the main result of
\cite{BR80} in order to get it in the form of Theorem
\ref{T:BourgainRos}. (I decided not to reproduce the whole proof
because this would lead to too much copying.)\medskip

We describe the modifications needed to get a proof of Theorem
\ref{T:BourgainRos} out of the proof following the Example 5.30 in
\cite{BL00}.

\begin{enumerate}

\item We replace $2^{-N(m,\ep)}$ in the definition of
$\delta(F,\ep)$ by $1/k_{\lceil N(m,\ep)\rceil}$, where $k_{\lceil
N(m,\ep)\rceil}$ is the corresponding term of the sequence
$\{k_n\}_{n=1}^\infty$.

\item We replace the assumption at the bottom of page 119 by the
assumption that there is a $E$-valued martingale $\{f_n\}$ on
$[0,1]$ adapted to some sequence $\{\mathcal{B}_n\}_{n=1}^\infty$
of finite $\sigma$-algebras satisfying $|\mathcal{B}_n|\le k_n$
and bounded in the sense that $\sup_n||f_n||_\infty\le 1$

\item In the second paragraph on page 120 we replace the
corresponding sentences by: The value of $g_m(t)-f_m(t)$ is a
convex combination of at most $k_s$ values of $g_s-f_s$ on the
$\mathcal{B}_m$-atom containing $t$. Using the estimate
$\mu\{|h|\ge\delta\}\le k\max\mu\{|h_i|\ge\delta\}$ whenever $h$
is a \emph{convex combination} of the $k$ functions $\{h_i\}$, we
see that for every $0\le t\le 1$
\[d(g_m(t),f_m(t))\le k_s\delta(F_n,\ep_n)\le\gamma(F_n,\ep_n).\]

\end{enumerate}

The rest of the proof is the same.

\begin{proof}[Proof of Theorem \ref{T:NotTestSp}] Theorem
\ref{T:BourgainRos} is not immediately applicable to the
martingales which we construct in Theorems \ref{T:MetrNonRNP} and
\ref{T:IsoMNonRNP}. There are two problems with its applicability:

\begin{itemize}

\item[{\bf (1)}] In our construction we can claim the norms
$||M_{2n}-M_{2n-1}||$ are bounded away from $0$, but there are no
lower bounds on $||M_{2n+1}-M_{2n}||$.

\item[{\bf (2)}] Another problem is that the sequence of finite
$sigma$-algebras $\{\mathcal{B}_n\}_{n=0}^\infty$ which we get in
our construction depends not only on the metric space which we
consider, but also on the embedding and the Banach space into
which we embed.

\end{itemize}

We can easily find a way around obstacle {\bf (1)}. We just
consider the martingale $\{M_{2n-1}\}_{n=1}^\infty$. Since the
conditional expectation is a contraction on $L_1(Y)$ even in the
Banach-space-valued case (see \cite[Section 1.1]{Pis11}), applying
the conditional expectation with respect to $\mathcal{B}_{2n}$ to
$M_{2n+1}-M_{2n-1}$ we get
\[||M_{2n+1}-M_{2n-1}||_1\ge||M_{2n}-M_{2n-1}||_1.\]
Therefore the norms of all differences of the martingale
$\{M_{2n-1}\}_{n=1}^\infty$ are bounded away from $0$ in $L_1(Y)$.
\medskip

Now we analyze obstacle {\bf (2)}. Let $X$ be any metric space
satisfying the conditions of Theorem \ref{T:MetrNonRNP} or
\ref{T:IsoMNonRNP}. By the corresponding proof this implies that
for any bilipschitz map $f:X\to Y$ into a Banach space $Y$ we can
construct a bounded $Y$-valued martingale $\{M_n\}$ such that
$\inf_n||M_{2n}-M_{2n-1}||\ge\delta>0$. By the previous remark we
can create out of it a new martingale $S_k$ with
$\inf_k||S_{k}-S_{k-1}||\ge\delta>0$. Now about the corresponding
$\sigma$-algebras. The $\sigma$-algebras corresponding to $M_0$,
$M_1$ and $M_2$ depend only of the metric space $X$ and the choice
of points which we make in the first step. The $\sigma$-algebras
$\mathcal{B}_3$ and $\mathcal{B}_4$ depend also on the choices of
$z_i$ and $\widetilde z_i$ which we make in the first step of the
construction. Therefore there are many different choices for the
algebras $\mathcal{B}_3$ and $\mathcal{B}_4$, the choice that we
have to make depends on the space $Y$ and the embedding $f$.
Important point is that there are finitely many options (because
there are finitely many ways to choose different collections
$\{z_i'\}$). Thus we have shown the following claim.
\medskip

\begin{claim} For each $n\in\mathbb{N}$ there exist finitely many finite $\sigma$-algebras
$\{\mathcal{B}_{n,j}\}_{j=1}^{m(n)}$ such that for each Banach
space $Y$ and each bilipschitz embedding of $M$ into $Y$ there is
a bounded divergent martingale $\{S_n\}_{n=1}^\infty$ with respect
to a filtration $\{\mathcal{B}_n\}_{n=1}^\infty$, where each
$\mathcal{B}_n$ is one of $\{\mathcal{B}_{n,j}\}_{j=1}^{m(n)}$,
and such that $\sup_n||S_n||_\infty<\infty$ and
$\inf_n||S_{n+1}-S_{n}||_1\ge\delta>0$.
\end{claim}

It remains to let
\[k_n=\max_{j\in\{1,\dots,m(n)\}}|\mathcal{B}_{n,j}|\] and to apply
Theorem \ref{T:BourgainRos}.
\end{proof}

\section{A submetric test space for reflexivity}\label{S:SubMetrRefl}

The purpose of this section is to show that the well-known
(linear) characterization of reflexivity leads to a submetric test
space characterization of reflexivity.\medskip

Let $\Delta \ge 1$. The submetric space $X_\Delta$ is the space
$\ell_1$ with its usual metric. The only thing which makes it
different from $\ell_1$ is the set of active pairs $S_\Delta$: A
pair $(x,y)\in X_\Delta\times X_\Delta$ is active if and only if
\begin{equation}\label{E:DefXDelta} ||x-y||_1\le\Delta||x-y||_s,\end{equation} where $||\cdot||_s$
is the summing norm, that is,

\[||\{a_i\}_{i=1}^\infty||_s=\sup_k\left|\sum_{i=1}^k
a_i\right|.\]

\begin{theorem}\label{T:SubMetRef} $X_\Delta$, $\Delta\ge 2$ is a submetric test space
for reflexivity.
\end{theorem}

We start by recalling the characterization of reflexivity
developed in a series of papers around 1960:
\cite{Pta59,Sin62,Pel62,Jam64b,MM65}. We state it in the following
way (we use the standard terminology of \cite{LT77}):

\begin{theorem}\label{T:NonRefl} For each $0<\theta<1$ there exists $1<B<\infty$ such that a Banach space $Y$ is
non-reflexive if and only if there is a basic sequence
$\{y_i\}_{i=1}^\infty\subset Y$ with basic constant $B$ and
$||y_i||=1$; and a functional $f\in Y^*$ such that $||f||=1$ and
$f(y_i)=\theta$ for all $i\in\mathbb{N}$.
\end{theorem}

\begin{proof}[Proof of Theorem \ref{T:SubMetRef}] Suppose that $Y$ is nonreflexive and show that in such a case $\ell_1$ admits a \emph{partially} bilipschitz embedding
with the set of active pairs $S_\Delta$. Let $\{y_i\}\subset Y$
and $f\in Y^*$ satisfy conditions of Theorem \ref{T:NonRefl}. We
embed $X_\Delta$ into $Y$ in the following way: we map a unit
vector $e_i$ of $\ell_1$ to $y_i$, and extend this map by
linearity. It is clearly a Lipschitz map (on the whole space
$\ell_1$). \medskip

We need to estimate $||u-v||_X$ from below in terms of $||u-v||_1$
for an active pair $(u,v)$. We let $u=\sum_iu_iy_i$,
$v=\sum_iv_iy_i$. We have

\[\begin{split}
||u-v||_X&\ge\sup_k\frac{||\sum_{i=1}^k(u_i-v_i)y_i||}{B}\ge
\sup_k\frac{|f(\sum_{i=1}^k(u_i-v_i)y_i)|}{B}\\
&=\theta\sup_k\frac{|\sum_{i=1}^k(u_i-v_i)|}{B}=\frac{\theta}{B}||u-v||_s\ge\frac{\theta}{B\Delta}||u-v||_1.
\end{split}\]

Now let $T:\ell_1\to Y$ be a partially bilipschitz embedding with
the set of active pairs $S_\Delta$ and constant $C$ (see
\eqref{E:PartBilipsch}). Observe that each vector in $z\in\ell_1$
can be represented as a difference of two vectors $z=x_1-x_2$ with
positive coordinates, for which $||x_1||_1=||x_1||_s$,
$||x_2||_1=||x_2||_s$ and $||z||_1=||x_1||_1+||x_2||_1$. Therefore
a partially bilipschitz embedding of $X_\Delta$ into $Y$ is a
Lipschitz map of $\ell_1$ into $Y$:
\begin{equation}\label{E:LipEvr}\begin{split}||T(y+z)-T(y)||_Y&\le
||T(y+z)-T(y+x_1)||_Y+||T(y+x_1)-Ty||_Y\\&\le
C||x_2||_1+C||x_1||_1\\&= C||z||_1.\end{split}\end{equation}

If $Y$ does not have the RNP, then $Y$ is nonreflexive (as is well
known), see \cite{BL00}, and there is nothing to prove in this
case. If $Y$ has the RNP, then, by the theorem of Aronszajn
\cite{Aro76}, Christensen \cite{Chr73}, and Mankiewicz
\cite{Man73} (see \cite[Theorem 6.42]{BL00}), there is a point $p$
of G\^{a}teaux differentiability of $T$ in $\ell_1$. Let $D$ be
the G\^{a}teaux derivative of $T$ at $p$. Since $T$ is partially
bilipschitz, for each $z\in\ell_1$ satisfying
$||z||_1\le\Delta||z||_s$ we have \[||z||_1\le
||T(p+z)-T(p)||_Y\le C||z||_1.\] Using this inequality,
\eqref{E:LipEvr}, and the definition of the G\^{a}teaux derivative
we get that $D:\ell_1\to Y$ is a linear operator satisfying
$||D||\le C$ and
\begin{equation}\label{E:ConsPartBil} ||Dz||_Y\ge ||z||_1 \hbox{ for
each } z\in\ell_1 \hbox{ satisfying
}||z||_1\le\Delta||z||_s.\end{equation}

Let $u_i=De_i$, where $\{e_i\}$ is the unit vector basis of
$\ell_1$. Then the sequence $\{u_i\}$ satisfies the conditions:

\begin{enumerate}

\item $1\le ||u_i||\le C$.

\item\label{I:Conv} $\forall
i\in\mathbb{N}\quad\dist(\conv(u_1,\dots,u_i),\conv(u_{i+1},\dots))\ge
2$.

\end{enumerate}

In fact, the first condition follows from $||D||\le C$. The second
condition follows from the inequality
\[\left\|\sum_{j=1}^i\alpha_je_j-\sum_{j=i+1}^\infty\alpha_je_j\right\|_1\le
2\left\|\sum_{j=1}^i\alpha_je_j-\sum_{j=i+1}^\infty\alpha_je_j\right\|_s,
\]
satisfied whenever $\alpha_j$ are nonnegative and satisfy
$\sum_{j=1}^i\alpha_j=1$ and $\sum_{j=i+1}^\infty\alpha_j=1$.
Therefore (recall that $\Delta\ge 2$)
\[\begin{split}\left\|D\left(\sum_{j=1}^i\alpha_je_j-\sum_{j=i+1}^\infty\alpha_je_j\right)\right\|_Y
&=\left\|\sum_{j=1}^i\alpha_ju_j-\sum_{j=i+1}^\infty\alpha_ju_j\right\|_Y\\&\stackrel{\eqref{E:ConsPartBil}}{\ge}
\left\|\sum_{j=1}^i\alpha_je_j-\sum_{j=i+1}^\infty\alpha_je_j\right\|_1=2.\end{split}
\]

It is well known (and is a version of the characterization of
reflexivity developed in \cite{Pta59,Sin62,Pel62,Jam64b,MM65}; see
\cite[pp.~49--55]{Bea82} and \cite[Chapter 6]{Ost13b}) that the
existence of such sequence $\{u_j\}$ in $Y$ implies nonreflexivity
of $Y$.\end{proof}

\end{large}

\begin{small}

\renewcommand{\refname}{
\end{small}

\end{document}